\documentclass[10pt]{article}
\usepackage{color}
\usepackage{amssymb}
\usepackage{amsthm,array,amssymb,amscd,amsfonts,latexsym, url}
\usepackage{amsmath}

\newtheorem{theorem}{Theorem}[section]
\newtheorem{proposition}[theorem]{Proposition}
\newtheorem{lemma}[theorem]{Lemma}
\newtheorem{corollary}[theorem]{Corollary}
\newtheorem{remark}[theorem]{Remark}

\newtheorem{variant}[theorem]{Variant}

\title{Rational equivalence of $0$-cycles on $K3$ surfaces and  conjectures of Huybrechts and O'Grady}
\author{Claire Voisin
\\CNRS, \'{E}cole Polytechnique}
\date{}

\newfont{\gothic}{eufb10}
\begin{document}
\maketitle

\begin{flushright}
{\it Pour Rob Lazarsfeld, \`{a} l'occasion\\ de son soixanti\`{e}me anniversaire}
\end{flushright}
\begin{abstract}
We give a new interpretation of O'Grady's filtration on the
$CH_0$ group
 of a $K3$ surface. In particular, we get a new characterization of the
 canonical $0$-cycles $kc_X$ :  given $k\geq0$, $kc_X$ is the only $0$-cycle of degree $k$ on $X$
 whose orbit under rational equivalence is of dimension
 $k$. Using this, we extend results of Huybrechts and O'Grady
 concerning Chern classes of simple vector bundles on $K3$ surfaces.
 \end{abstract}

\section{Introduction}
Let $X$ be a  projective $K3$ surface. In \cite{beauvoisin}, Beauville and the author
proved that $X$ carries a canonical $0$-cycle $c_X$ of degree $1$,
which is the class in $CH_0(X)$ of any point of $X$ lying on a (possibly singular)
rational curve on $X$.
This cycle has the property that
for any divisors $D,\,D'$ on $X$,
we have
$$D\cdot D'={\rm deg} (D\cdot D')\,c_X\,\,{\rm in}\,\,CH_0(X).$$
In recent works of Huybrechts  \cite{huybrechts} and O'Grady \cite{ogrady}, this
$0$-cycle appeared to have other characterizations. Huybrechts proves for example the following result (which is proved in \cite{huybrechts}
to have much more general consequences on  spherical objects and autoequivalences of the derived category of $X$):
\begin{theorem}\label{theohuybrechts}(Huybrechts \cite{huybrechts})
 Let $X$ be a projective complex
 $K3$ surface. Let ${F}$ be a  simple vector bundle on $X$ such that $H^1( {\rm End}\,F)=0$ (such an $F$ is called spherical in \cite{huybrechts}).
  Then $c_2({F})$ is
proportional to $c_X$ in $CH_0(X)$ if one of the following conditions holds.
\begin{enumerate}
\item  The Picard number of $X$ is at least $2$.

\item  The Picard group of $X$ is $\mathbf{Z}H$ and the determinant of $F$ is
equal to $k H$ with $k=\pm1\,\,{\rm mod}. \,\,r ={\rm rank}\,F$.
\end{enumerate}
\end{theorem}

This result is extended in the following way by O'Grady : In \cite{ogrady}, he
introduces the following increasing filtration of $CH_0(X)$
$$S_0(X)\subset S_1(X)\subset \ldots \subset  S_d(X)\subset\ldots\subset CH_0(X),$$
where $S_d(X)$ is defined as the set of classes of  cycles of the form
$z+z'$, with $z$ effective of degree $d$ and $z'$ a multiple of
$c_X$. It is also convenient to
introduce $S_d^k(X)$ which will be by definition the set of degree $k$ $0$-cycles
on $X$ which lie in $S_d(X)$. Thus by definition
$$S_d^k(X)=\{z\in CH_0(X),\,z=z'+(k-d)c_X\},$$
where $z'$ is effective of degree $d$.

 Consider a torsion free or more generally a pure  sheaf $\mathcal{F}$ on $X$ which is $H$-stable with respect to a polarization
$H$. Let $2d(v_{\mathcal{F}})$ be the dimension of the space of
deformations of $\mathcal{F}$, where $v_\mathcal{F}$ is the Mukai
vector of $\mathcal{F}$ (cf. \cite{huybrechtslehn}). We recall that
$v_\mathcal{F}\in H^*(X,\mathbf{Z})$ is the triple
$$(r,l,s)\in H^0(X,\mathbf{Z})\oplus H^2(X,\mathbf{Z})\oplus H^4(X,\mathbf{Z}),$$ with $r={\rm rank}\,\mathcal{F}$, $l=c_1^{top}({\rm det}\,\mathcal{F})$
 and $s\in H^4(X,\mathbf{Z})$ is defined
as
\begin{eqnarray}v_\mathcal{F}=ch(\mathcal{F})\sqrt{td(X)}.\label{vraimukai}
\end{eqnarray} In particular
we get by the Riemann-Roch formula that
$$\sum_i(-1)^i{\rm dim}\,Ext^i(\mathcal{F},\mathcal{F})=<v_\mathcal{F},v_\mathcal{F}^*>=2rs-{l^2}=2-2d(v_\mathcal{F}),$$
where $<\,,\,>$ is the intersection pairing on $H^*(X,\mathbf{Z})$,
and $v^*=(r,-l,s)$ is the Mukai vector of $\mathcal{F}^*$ (if $\mathcal{F}$ is locally free).

In particular $d(v_\mathcal{F})=0$ if $\mathcal{F}$ satisfies
${\rm End}\,\mathcal{F}=\mathbb{C}$ and ${\rm Ext}^1(\mathcal{F},\mathcal{F})=0$, so that $\mathcal{F}$ is spherical as in Huybrechts' theorem.
Noticing that $S_0(X)=\mathbf{Z}c_X$, one can then rephrase Huybrechts' statement by saying
that if $\mathcal{F}$ satisfies ${\rm End}\,(\mathcal{F})=\mathbb{C},\,d(v_\mathcal{F})=0$, then
$c_2(\mathcal{F})\in S_0(X)$, assuming the Picard number of $X$ is at least $2$.

O'Grady then extends Huybrechts' results as follows:
\begin{theorem} \label{theogrady} (O'Grady \cite{ogrady}) Assuming $\mathcal{F}$ is $H$-stable, one has $c_2(\mathcal{F})\in S_{d(v_\mathcal{F})}(X)$, $v_\mathcal{F}=(r,l,s)$,
if furthermore one of the following conditions holds:

\begin{enumerate}

\item
 $l=H$, $l$ is primitive  and  $s\geq 0$.

\item
The Picard number of $X$ is at least $2$, $r$ is  coprime to the divisibility of $l$ and $H$ is $v$-generic.

\item

$r\leq 2$ and moreover $H$ is $v$-generic if $r=2$.
\end{enumerate}
\end{theorem}

In fact, O'Grady's result is stronger, as he also shows that $S_{d(v)}^{k}(X),\,k={\rm deg}\,c_2(v)$,
is equal to the set of classes $c_2(\mathcal{G})$ with $\mathcal{G}$ a deformation of
$\mathcal{F}$. O'Grady indeed proves, by a nice argument involving
the rank of the  Mukai holomorphic $2$-form on the moduli space of deformations of
$\mathcal{F}$, the following result:
\begin{proposition} \label{propogrady}(O'Grady \cite[Prop. 1.3]{ogrady}) If there is a
$H$-stable torsion free sheaf  $\mathcal{F}$ with $ v=v(\mathcal{F})$, and the conclusion of Theorem
\ref{theogrady} holds for the deformations of $\mathcal{F}$, then
$$\{c_2(\mathcal{G}),\,\mathcal{G}\in \overline{\mathcal{M}^{st}}(X,H,v) \,\,\}= S_{d(v)}^k(X),\,k={\rm deg}\,c_2(\mathcal{F}).$$
\end{proposition}
In this statement, $\overline{\mathcal{M}^{st}}(X,H,v)$ is any smooth completion of the
moduli space of $H$-stable sheaves with Mukai vector $v$.

Our results in this paper are of two kinds: First of all we provide another description of $S_d^k(X)$ for any $d\geq 0$, $k\geq d$.
In order to state this result, let us introduce the following
notation: Given an integer $k\geq 0$, and a cycle $z\in CH_0(X)$ of degree $k$,
the subset $O_z$ of $X^{(k)}$ consisting
 of effective cycles $z'\in X^{(k)}$ which are rationally equivalent to
$z$ is a countable union of closed algebraic subsets of $X^{(k)}$
(see \cite[Lemma 10.7]{voisinbook}). This is the ``effective orbit''
of $z$ under rational equivalence, and the analogue of $|D|$ for a
divisor $D\in CH^1(W)$ on any variety $W$. We define ${\rm
dim}\,O_z$ as the supremum of the dimensions of the components of
$O_z$. This is the analogue of $r(D)={\rm dim}\,|D|$ for a divisor
$D\in CH^1(W)$ on any variety $W$. We will prove the following:

\begin{theorem}\label{char} Let $X$ be a projective
$K3$ surface. Let $k\geq d\geq 0$. We have the following characterization of $S_d^k(X)$:
$${\rm if}\,\,k>d,\,\,S_d^k(X)=\{z\in CH_0(X),\,O_z\,\,{\rm nonempty},\,\,{\rm dim}\,O_z\geq k-d \}.$$

\end{theorem}
\begin{remark}{\rm The inclusion
$S_d^k(X)\subset\{z\in CH_0(X),\,\,O_z\,\,{\rm nonempty},\,\,{\rm dim}\,O_z\geq k-d \}$ is easy since
the cycle $(k-d)c_X$ has its orbit of dimension $\geq k-d$, (for example
$C^{(k-d)}\subset X^{(k-d)}$, for  any rational curve $C\subset X$, is contained in the orbit of $(k-d)c_X$). Hence any cycle of the form $z+(k-d)c_X$ with $z$ effective
of degree $d$ has an orbit of dimension $\geq k-d$.
}
\end{remark}
A particular case of the theorem above is the case where $d(v)=0$.
By definition $S_0(X)$ is the subgroup $\mathbf{Z}c_X\subset
CH_0(X)$. We thus have:
\begin{corollary}  \label{corochar} For $k>0$, the cycle $kc_X$ is the unique  $0$-cycle $z$ of  degree
$k$ on $X$ such that ${\rm dim}\,O_z\geq k$.
\end{corollary}
\begin{remark}{\rm We have in fact ${\rm dim}\,O_z=k,\,z=kc_X$ since
by Mumford's theorem
\cite{mumford},
 any component $L$ of $O_z$ is Lagrangian for the holomorphic
 symplectic form on
 $S^{(k)}_{reg}$, hence of dimension $\leq k$ if $L$ intersects $S^{(k)}_{reg}$.
 If $L$ is contained in the singular locus of $S^{(k)}$, we can consider the
 minimal multiplicity-stratum of $S^{(k)}$  containing $L$, which is determined by
  the multiplicities $n_i$ of the general cycle
  $\sum_in_ix_i,\,x_i$ distinct, parametrized by $L$) and apply the same argument.}
 \end{remark}
 \begin{remark}{\rm We will give in Section \ref{sec1} an alternative proof of
 Corollary \ref{corochar}, using the remark above, and the fact that
 any Lagrangian subvariety of $X^k$ intersects  a product $D_1\times\ldots\times D_k$ of
 ample divisors on $X$.}
 \end{remark}
 Our main application of Theorem \ref{char} is the following result which
 generalizes O'Grady's and Huybrechts' theorems \ref{theogrady}, \ref{theohuybrechts}
  in the case of simple vector bundles
 (instead of semistable torsion free sheaves). We do not need any of the assumptions
 appearing in Theorems \ref{theogrady}, \ref{theohuybrechts}, but our results, unlike those of O'Grady,
 are restricted to the locally free case.
\begin{theorem}\label{ogrady} Let $X$ be a projective
$K3$ surface. Let $F$ be a simple vector bundle on $X$ with Mukai vector
$v=v(F)$. Then
$$c_2(F)\in S_{d(v)}(X).$$
\end{theorem}
A particular case of this statement is the case where $d=0$ : The corollary below proves Huybrechts' Theorem \ref{theohuybrechts} without any assumption on the Picard group of the $K3$ surface or on the determinant of $F$. It is conjectured in \cite{huybrechts}.
\begin{corollary} Let $F$ be a simple rigid vector bundle on a $K3$ surface. Then
the element $c_2(F)$ of $ CH_0(X)$ is a multiple of $c_X$.
\end{corollary}
We also deduce  the following corollary, in the same spirit (and with essentially the same proof)
as Proposition \ref{propogrady}:
\begin{corollary}\label{corogrady}
Let $v\in H^*(X,\mathbf{Z})$ be a Mukai vector, with $k=c_2(v)$. Assume there exists a simple vector bundle
$F$
on $X$ with Mukai vector $v$. Then
$$S_d^k(X)= \{c_2({G}), \,\,G\,\,{\rm a\,\, simple\,\, vector\,\, bundle\,\,on}\,\,X,\,\,v_G=v\},$$
where $k=c_2(v):=c_2(F)$.

\end{corollary}
These results answer for
 simple vector bundles on $K3$ surfaces questions asked  by  O'Grady  (see \cite[section 5]{ogrady}) for simple sheaves.

The paper is organized as follows: in Section \ref{sec1}, we prove Theorem \ref{char}. We also show a variant concerning family of subschemes (rather than $0$-cycles) of given length in a constant rational equivalence class.
In section \ref{sec2}, Theorem \ref{ogrady} and Corollary \ref{corogrady}
are  proved.

\vspace{0.5cm}

{\bf Thanks.} I thank Daniel Huybrechts and Kieran
O'Grady for useful and interesting comments on a preliminary version of this paper.

\vspace{0.5cm}

{\it  J'ai grand plaisir \`{a} d\'{e}dier cet article \`{a} Rob Lazarsfeld, avec
toute mon estime et ma sympathie.
Son merveilleux
article  \cite{laz} red\'{e}montrant un grand th\'{e}or\`{e}me classique sur les
s\'{e}ries lin\'{e}aires sur les courbes g\'{e}n\'{e}riques
a aussi jou\'{e} un r\^{o}le d\'{e}cisif dans l'\'{e}tude des fibr\'{e}s vectoriels et des $0$-cycles
sur les surfaces $K3$.}
\section{An alternative description of O'Grady's filtration\label{sec1}}
This section is devoted to the proof of Theorem \ref{char} which we state in the following form:
\begin{theorem}\label{autreforme}
Let $k\geq d$ and let $Z\subset X^{(k)}$ be a Zariski closed
irreducible algebraic subset of dimension $k-d$. Assume that all
cycles of $X$ parametrized by $Z$ are rationally equivalent in $X$.
Then the class of these cycles belongs to $S_d^k(X)$.
\end{theorem}

We will need for the proof the following simple lemma, which already appears in
\cite{voisinGandT}.
\begin{lemma}\label{lemma} Let $X$ be a  projective  $K3$ surface and let $C\subset S$ be a
(possibly singular) curve such that all points of $C$ are rationally equivalent
in $X$. Then any point of $C$ is rationally equivalent to $c_X$.
\end{lemma}
\begin{proof} Let $L$ be an ample line bundle on $X$.
Then $c_1(L)_{\mid C}$ is a $0$-cycle on $C$ and our assumptions imply
that $j_*(c_1(L)_{\mid C})={\rm deg}\,(c_1(L)_{\mid C})\,c$, for any point $c$ of $C$.

On the other hand,  we have
$$j_*(c_1(L)_{\mid C})=c_1(L)\cdot C\,\,{\rm in}\,\,CH_0(X)$$
and thus, by \cite{beauvoisin}, $j_*(c_1(L)_{\mid C})={\rm deg}\,(c_1(L)_{\mid C})\,c_X$ in $CH_0(X)$.
Hence we have
$${\rm deg}\,(c_1(L)_{\mid C})\,c={\rm deg}\,(c_1(L)_{\mid C})\,c_X\,\,{\rm in}
\,\,CH_0(X).$$
This concludes the proof, since $c$ is arbitrary, ${\rm deg}\,(c_1(L)_{\mid C})\not=0$ and $CH_0(X)$ has no torsion.
\end{proof}
\begin{lemma} \label{corodense} The union of
curves $C$ satisfying the property stated in Lemma
\ref{lemma} is Zariski dense in $X$.

\end{lemma}
\begin{proof} The $0$-cycle $c_X$ is represented by any point lying on a (singular)
rational curve
$C\subset X$ (see \cite{beauvoisin}), so the result is clear if one knows that there
are infinitely many distinct rational curves contained in $X$. This result
is to our knowledge known only for general $K3$ surfaces but not for all
$K3$ surfaces (see however \cite{hassettetal}
for results in this direction). In any case, we can use the following argument
which already appears in \cite{maclean}:
By \cite{morimukai}, there is a $1$-parameter family of (singular) elliptic curves
$E_t$ on $X$. Let $C$ be a rational curve on $X$ which meets the fibers $E_t$.
For any integer $N$, and any point $t$,  consider the points
$y\in \widetilde{E}_t$ (the desingularization of $E_t$), which
are rationally equivalent in $\widetilde{E}_t$ to the sum of a point
$x_t\in E_t\cap C$ (hence rationally equivalent to $c_X$)
and a $N$-torsion $0$-cycle on $\widetilde{E}_t$.

As $CH_0(X)$ has no torsion, the images  $y_t$ of these points in
$X$ are all rationally equivalent to $c_X$ in $X$. Their images  are
clearly parametrized  for $N$ large enough by a (maybe reducible)
curve $C_N\subset X$. Finally, the union over all $N$ of the points
$y_t$ above is Zariski dense in each $\widetilde{E}_t$, hence the
union of the curves $C_N$ is Zariski dense in $X$.

\end{proof}

\begin{proof}[Proof of Theorem \ref{autreforme}] The proof is by induction on $k$, the case $k=1$, $d=0$ being Lemma \ref{lemma} (the case $k=1$, $d=1$ is trivial).
Let $Z'$ be an irreducible component of the inverse image of
$Z$ in $X^k$. Let $p:Z'\rightarrow X$ be the first projection.
We distinguish two cases and note that they exhaust all possibilities, up to
replacing $Z'$ by another component $Z''$ deduced from $Z'$ by letting the symmetric group $\mathfrak{S}_k$ act.

\vspace{0.5cm}

{\it Case 1.} {\it The morphism $p:Z'\rightarrow X$ is dominant.}
For a curve $C\subset X$ parametrizing points rationally equivalent
to $c_X$, consider the hypersurface
$$Z'_C:=p^{-1}(C)\subset Z'.$$
Let $q:Z'\rightarrow X^{k-1}$ be the projection on the product of
the $k-1$ last factors. Assume first that ${\rm dim}\,q(Z'_C)={\rm
dim}\,Z'_C=k-d-1$. Note that all cycles of $X$ parametrized by
$q(Z'_C)$ are rationally equivalent in $X$. Indeed, an element $z$
of $Z'_C$ is of the form $(c,z')$ with $c\in C$ so that $c=c_X$ in
$CH_0(X)$. So the rational equivalence class of $z'$ is equal to
$z-c_X$ and is independent of $z'\in Z'_C$. Thus the induction
assumption applies and  the cycles of degree $k-1$ parametrized by
${\rm Im}\,q$ belong to $S^{k-1}_d(X)$. It follows in turn  that the
classes of the cycles parametrized by $Z'$ (or $Z$) belong to
$S_d^k(X)$. Indeed, as just mentioned above, a $0$-cycle $z$
parametrized by $Z'$ is rationally equivalent to $z=c_X+z'$ where
$z'\in S^{k-1}_d(X)$, so $z'$ is rationally equivalent to
$(k-d-1)c_X+z''$, $z''\in X^{(d)}$. Hence $z$ is rationally
equivalent in $X$ to $(k-d)c_X+z''$, for some $z''\in X^{(d)}$. Thus
$z\in S^k_d(X)$.

Assume to the contrary that ${\rm dim}\,q(Z'_C)<{\rm dim}\,Z'_C=k-d-1$ for any curve
$C$ as above. We use now the fact (see Lemma \ref{corodense})
that these curves $C$ are Zariski dense in $X$.
We can thus assume that there is a point
$x\in Z'_C$ which is generic in $Z'$, so that both $Z'$ and $Z'_C$ are smooth at $x$,
of respective dimensions $k-d$ and $k-d-1$.
The fact that ${\rm dim}\, q(Z'_C)<k-d-1$ implies that $q$ is
not of maximal rank $k-d$ at $x$ and as $x$ is generic in $Z'$, we conclude that
$q$ is of rank $<k-d$ everywhere on $Z'_{reg}$, so that
${\rm dim}\,{\rm Im}\,q\leq k-d-1$.

Now recall that all $0$-cycles parameterized by $Z'$ are rationally equivalent.
It follows that for any fiber $F$ of $q$, all points in $p(F)$ are rationally equivalent in $X$. This implies that all these points are rationally equivalent
to $c_X$ by Lemma \ref{lemma}.
This contradicts the fact that $p$ is surjective.

\vspace{0.5cm}

{\it Case 2.} {\it None of the projections $pr_i,\,i=1,\ldots, k$,
from $X^k$ to its factors is dominant.} Let $C_i:={\rm
Im}\,p_i\subset X$ if ${\rm Im}\,pr_i$ is a curve, and any curve
containing ${\rm Im}\,p_i$ if ${\rm Im}\,p_i$ is a point. Thus $Z'$
is contained in $C_1\times\ldots\times C_k$.

Let $C$ be a non necessarily irreducible ample curve such that all points
in $C$ are rationally equivalent to $c_X$.
Observe that the line bundle
$pr_1^*\mathcal{O}_X(C)\otimes \ldots\otimes pr_k^*\mathcal{O}_X(C)$ on $X^k$
has its restriction to
 $C_1\times\ldots\times C_k$ ample and that its $k-d$-th self-intersection
on $C_1\times\ldots\times C_k$ is  a complete intersection of ample
divisors and is equal to
\begin{eqnarray}
\label{cyclepositif}
W:=(k-d)!\sum_{i_1<\ldots<i_{k-d}}p_{i_1}^*\mathcal{O}_{C_1}(C)\cdot \ldots\cdot p_{i_{k-d}}^*\mathcal{O}_{C_{k-d}}(C)
\end{eqnarray}
in
$CH^{k-d}(C_1\times \ldots\times C_k)$,
where the $p_i$ are the projections from $\prod_iC_i$ to its factors.

The cycle $W$ of (\ref{cyclepositif}) is as well
the restriction to $C_1\times\ldots \times C_k$
of the effective cycle
\begin{eqnarray}
\label{cycleeffectif}
W':=(k-d)!\sum_{i_1<\ldots<i_{k-d}}pr_{i_1}^*C\cdot\ldots\cdot pr_{i_{k-d}}^*C.
\end{eqnarray}
As
 the $k-d$ dimensional subvariety $Z'$ of $C_1\times\ldots\times C_k$
  has a nonzero intersection with $W$, it follows that
  the intersection number of
  $Z'$ with $W'$ is nonzero in $X^k$, hence
  that
  $$Z'\cap pr_{i_1}^*C\cdot\ldots\cdot pr_{i_{k-d}}^*C\not=\emptyset$$
  for some choice of indices $i_1<\ldots<i_{k-d}$.
  This means that there exists a cycle in $Z$ which is of the form
  $$z=z'+z''$$
  with  $z'\in C^{(k-d)}$ and $z''\in X^{(d)}$. As $z'$ is supported on $C$,
  it is equal to $(k-d)c_X$ in $CH_0(X)$ and  we conclude that
  $z\in S_d^k(X)$.

\end{proof}
Let us now prove the following variant of Theorem
\ref{autreforme}. Instead of a family of $0$-cycles (that is, elements of $X^{(k)}$),
we now consider families of $0$-dimensional {\it subschemes} (that is, elements of
$X^{[k]}$) :
\begin{variant}\label{variant} Let $k\geq d$ and let $Z\subset X^{[k]}$ be a Zariski closed irreducible
algebraic subset of dimension $k-d$. Assume
that all cycles of $X$ parametrized by $Z$ are rationally equivalent
in $X$.
Then the class of these cycles belongs to $S_d^k(X)$.
\end{variant}
\begin{proof} Let $z\in Z$ be a general point. The cycle $c(z)$ of  $z$,
where $c:X^{[k]}\rightarrow X^{(k)}$ is the Hilbert-Chow morphism , is of the
form $\sum_ik_i x_i$, with $\sum_ik_i=k$, where $x_i$ are $k'$ distinct points of $X$.
We have of course
\begin{eqnarray}\label{etdun}k'=k-\sum_i(k_i-1).
\end{eqnarray}

The fiber of $c$ over a cycle of the form $\sum_ik_i x_i$ as above is
of dimension $\sum_i(k_i-1)$ (see for example \cite{ellingsrud}). It follows that
the image $Z_1$ of $Z$ in $X^{(k)}$ is of dimension $\geq k-d-\sum_i(k_i-1)$.
By definition, $Z_1$ is contained in a multiplicity-stratum of $X^{(k)}$ where
the support of the considered cycles has cardinality $\leq k'$.
Let $Z'_1\subset X^{k'}$ be the set of $(x_1,\ldots,x_{k'})$ such that
$\sum_ik_ix_i\in c(Z)$. Then the morphism
$$Z'_1\rightarrow Z_1,\,(x_1,\ldots,x_{k'})\mapsto\sum_ik_ix_i$$
is finite and surjective, so that
\begin{eqnarray}\label{etdeux}{\rm dim}\,Z'_1= {\rm dim}\,Z_1\geq k-d-\sum_i(k_i-1),
\end{eqnarray}
which by (\ref{etdun}) can be rewritten as
$${\rm dim}\,Z'_1= {\rm dim}\,Z_1\geq k'-d.$$

Note that by construction, $Z'_1$ parameterizes $k'$-uples $(x_1,\ldots,x_{k'})$
with the property that
$\sum_ik_ix_i$ is rationally equivalent to a constant cycle.

The proof of the variant \ref{variant} then concludes with the following statement :
\begin{proposition} \label{prop4aout}Let $l$ be a positive integer,
$k_1>0,\ldots,k_{l}>0$ be positive multiplicities.
Let $Z$ be a  closed algebraic subset of
$X^{l}$. Assume that ${\rm dim} \,Z\geq l-d$ and
the cycles $\sum_ik_ix_i$, $(x_1,\ldots,x_{l})\in Z$, are all rationally equivalent in $X$. Then the class
of the cycles $\sum_ik_ix_i$, $(x_1,\ldots,x_{l})\in Z$, belongs to
$S_d^k(X)$, where $k=\sum_ik_i$.

\end{proposition}

\end{proof}
 For the proof of Proposition \ref{prop4aout}, we have to start with the following
Lemma:
\begin{lemma} \label{lemma4aout} Let $x_1,\ldots,x_d\in X$ and let $k_i\in \mathbf{Z}$.
Then $\sum_ik_ix_i\in S_d^k(X)$, $k=\sum_ik_i$.
\end{lemma}
\begin{proof} We use the following characterization of $S_d(X)$ given
by O'Grady: \begin{proposition}\label{procharogrady} (O'Grady
\cite{ogrady}) A cycle $z\in CH_0(X)$ belongs to $S_d(X)$ if and
only if there exists a (possibly singular, possibly reducible) curve
$j:C\subset X$, such that the genus of the desingularization of $C$
(or the sum of the genera of its components if $C$ is reducible) is
non greater than $d$ and $z$ belongs to ${\rm
Im}\,(j_*:CH_0(C)\rightarrow CH_0(X))$.

\end{proposition}
Let now $x_1,\ldots, x_d$ be as above. There exists by \cite{morimukai} a curve $C\subset X$, whose desingularization has
genus $\leq d$ and  containing $x_1,\ldots, x_d$. Thus for any $k_i$, the cycle
$\sum_ik_ix_i$ is supported on $C$, which proves the Lemma by Proposition
\ref{procharogrady}
\end{proof}
\begin{proof}[Proof of Proposition \ref{prop4aout}]
Proposition \ref{prop4aout}
is proved exactly
as Theorem\ref{autreforme}, by induction on $l$. In case 1 considered in the induction step, we apply the same argument as in that proof. In case 2
considered in the induction step, using the same notations as in that proof,  we conclude
that there is in $Z$ a  $l$-uple $(x_1,\ldots, x_l)$ satisfying (up to permutation of the indices)
$$x_{d+1}\,\ldots,x_{l}\in C,$$
and as any point of $C$ is rationally equivalent to
$c_X$, we find that
$$\sum_ik_ix_i=(\sum_{i>d}k_i)c_X+\sum_{l\leq i\leq d}k_ix_i.$$
By Lemma \ref{lemma4aout}, $\sum_{1\leq i\leq d}k_ix_i\in S_d(X)$, so that
$\sum_ik_ix_i\in S_d(X)$.

\end{proof}
As mentioned in the introduction, Theorem \ref{autreforme} in the case
 $d=0$ provides the following characterization of the cycle $kc_X$, $k>0$: It is the only degree $k$
 $0$-cycle $z$ of $X$, whose orbit $O_z\subset X^{(k)}$ is $k$-dimensional
 (cf. Corollary
 \ref{corochar}).
 Let us give a slightly more direct proof in this case.
 We use the following Lemma \ref{lag}:
  Let $V$ be a $2$-dimensional complex vector space. Let $\eta\in
 \bigwedge^2V^*$ be a nonzero generator, and let
 $\omega\in \bigwedge^{1,1}_\mathbf{R}(V^*)$ be a positive real $(1,1)$-form on $V$.
 \begin{lemma}\label{lag} Let $W\subset V^k$ be a $k$-dimensional complex vector subspace which is Lagrangian for the nondegenerate $2$-form
 $\eta_k:=\sum_ipr_i^*\eta$ on $V^k$, where the $pr_i$'s are the projections from $V^k$
 to $V$. Then $\prod_ipr_i^*\omega$ restricts to a volume form on
 $W$.

 \end{lemma}
 \begin{proof} The proof is by induction on $k$. Let $\pi:W\rightarrow V^{k-1}$ be the
 projector on the product of the last $k-1$ summands.
 We can clearly assume up to changing the order of factors, that ${\rm dim}\,{\rm Ker}\,\pi<2$. As
 ${\rm dim}\,{\rm Ker}\,\pi\leq 1$, we can choose a linear form $\mu$ on $V$ such that
   the $k-1$-dimensional
  vector space $W_\mu:={\rm ker}\,pr_1^*\mu_{\mid W}$ is sent injectively by
  $\pi$ to a $k-1$-dimensional subspace
  $W'$ of  $V^{k-1}$. Furthermore, since
  $W$ is Lagrangian for $\eta_k$, $W'$ is Lagrangian for
  $\eta_{k-1}$ because  $W_\mu\subset {\rm Ker}\,\mu\times V^{k-1}$,  and
  on ${\rm Ker}\,\mu\times V^{k-1}$, $\eta_k=\pi^*\eta_{k-1}$.
  By the induction hypothesis, the form
  $\prod_{i>1}pr_i^*\omega$ restricts to a volume form on
  $W'$, where the projections here are considered as restricted to $0\times V^{k-1}$,
  and it follows that
  $$ pr_1^*(\sqrt{-1}\mu\wedge\overline{\mu})\wedge \prod_{i>1}pr_i^*\omega$$ restricts to a volume form on
  $W$. It immediately follows that
  $\prod_{i\geq 1}pr_i^*\omega$ restricts to a volume form on
  $W$ since for a positive number $\alpha$, we have
  $$ \omega\geq  \alpha\sqrt{-1}\mu\wedge \overline{\mu}$$
  as real $(1,1)$-forms on $V$.

 \end{proof}
\begin{proof}[Proof of Corollary \ref{corochar}] Let $z\in CH_0(X)$ be a cycle  of degree $k$
 such that ${\rm dim}\,O_z\geq k$.
 Let $\Gamma\subset X^k$ be an irreducible component of the inverse
 image of a $k$-dimensional component of $O_z\subset X^{(k)}$ via the
 map $X^k\rightarrow X^{(k)}$.
 By Mumford's theorem
 \cite{mumford},  using the fact that all the $0$-cycles parameterized by $\Gamma$ are rationally equivalent in
 $X$, $\Gamma$ is Lagrangian for the symplectic form
 $\sum_ipr_i^*\eta_X$ on $X^k$, where $\eta_X\in H^{2,0}(X)$ is a generator.
 Let $L$ be an ample line bundle on $X$ such that there is a
 curve $D\subset X$ in the linear system $|L|$, all of whose components
 are rational.
 We claim that
 $$\Gamma\cap D^k\not=\emptyset.$$
 Indeed, it suffices to prove that the intersection number
 \begin{eqnarray}\label{intnumber}
 [\Gamma]\cdot [D^k]
 \end{eqnarray}
 is positive.
 Let $\omega_L\in H^{1,1}(X)$ be a positive representative
 of $c_1(L)$. Then
 (\ref{intnumber}) is equal to

 \begin{eqnarray}\label{integral}
 \int_{\Gamma_{reg}}\prod_ipr_i^*\omega_L.
 \end{eqnarray}
 By Lemma \ref{lag}, the form $\prod_ipr_i^*\omega_L$ restricts to a volume form
 on $\Gamma$ at any smooth point of $\Gamma$ and the integral (\ref{integral})
 is thus positive.
\end{proof}
\section{Second Chern class of simple vector bundles\label{sec2}}

This section is devoted to the proof of Theorem \ref{ogrady}. Recall first from
\cite{ogrady}, that in order to prove the result for a vector bundle $F$ on
$X$, it suffices to prove it for $F\otimes L$, where $L$ is a line bundle on
$X$. Choosing $L$ sufficiently ample, we can thus assume that
$F$ is generated by global sections, and furthermore that
\begin{eqnarray}\label{van}H^1(X,F^*)=0.
\end{eqnarray}
Let $r={\rm rank}\,F$. Choose a general $r-1$-dimensional subspace
$W$
 of $H^0(X,F)$,  and consider the evaluation morphism:
$$e_W:W\otimes\mathcal{O}_X\rightarrow F.$$
The following result is well-known
 (cf. \cite[5.1]{huybrechtslehn}).
\begin{lemma}  \label{necessary} The morphism $e_W$ is generically injective, and the locus $Z\subset X$
where its rank is $<r-1$ consists of $k$ distinct reduced points, where $k=c_2^{top}(F)$.

\end{lemma}
\begin{proof} Let $G=Grass(r-1,H^0(X,F))$ be the Grassmannian of
$r-1$-dimensional subspaces of $H^0(X,F)$. Consider the following
universal subvariety of $G\times X$:
$$G_{deg}:=\{(W,x)\in G\times X, {\rm rank}\,e_{W,x}<r-1\}.$$
Since $F$ is generated by sections, $G_{deg}$ is a fibration over $X$,
with fibers smooth away from the singular locus
$$G_{deg}^{sing}:=\{(W,x)\in G\times X, {\rm rank}\,e_{W,x}<r-2\}.$$
Furthermore, we have
$${\rm dim}\,G_{deg}={\rm dim}\,(G\times X)-2={\rm dim}\,G$$
and ${\rm dim}\,G_{deg}^{sing}<{\rm dim}\,G$.

Consider the first projection: $p_1:G_{deg}\rightarrow G$.
It follows from the observations above and from Sard's theorem that for general $W\in G$, $p_1^{-1}(W)$ avoids $G_{deg}^{sing}$
 and consists of finitely many reduced points
in $X$.
The statement concerning the number $k$ of points follows
from \cite[14.3]{fu}, or from the following argument
that we will need later on:
Given a $W$ such that the morphism $e_W$ is generically injective, and the locus
$Z_W$
where its rank is $<r-1$ consists of $k$ distinct reduced points,
we have an exact sequence
\begin{eqnarray}\label{exact}
0\rightarrow W\otimes \mathcal{O}_X\rightarrow F\rightarrow \mathcal{I}_{Z_W}\otimes \mathcal{L}\rightarrow0,\end{eqnarray}
where $\mathcal{L}={\rm det}\,F$.
Hence $c_2(F)=c_2(\mathcal{I}_Z\otimes \mathcal{L})=c_2(\mathcal{I}_Z)=Z$, and in particular
$c_2^{top}(F)={\rm deg}\,Z$.
This proves the lemma.
\end{proof}

By Lemma \ref{necessary}, we have a rational map
$$\phi:G\dashrightarrow X^{(k)},\,W\mapsto c(Z_W),$$
where $c:X^{[k]}\rightarrow X^{(k)}$ is the Hilbert-Chow morphism.

\begin{proposition} \label{propinj}If $F$ is simple and satisfies  the assumption (\ref{van}), the rational map $\phi$ is generically
one to one on its image.
\end{proposition}

\begin{proof} Let $G^0\subset G$
be the Zariski open set parameterizing the
subspaces $W\subset H^0(X,F)$ of dimension $r-1$ satisfying
the conclusions of Lemma \ref{necessary}. Note that $c$ is a local isomorphism at a point $Z_W$ of $X^{[k]}$ consisting
of $k$ distinct points,
so that the dimension of the image of $\phi$ is equal to the
dimension of the image of the rational map $G\dashrightarrow X^{[k]},\,W\mapsto Z_W$,
which we will also denote by $\phi$.
This $\phi$ is a morphism on $G^0$ and
it suffices to show
that the map $ \phi^0:=\phi_{\mid G^0}$ is injective.
Let $W\in G^0,\,Z:=\phi(W)$. For any $W'\in {\phi^0}^{-1}(Z)$, we have an exact
sequence as in (\ref{exact}):
\begin{eqnarray}\label{exact'}
0\rightarrow W'\otimes \mathcal{O}_X\rightarrow F\rightarrow \mathcal{I}_Z\otimes \mathcal{L}\rightarrow0,\end{eqnarray}
so that $W'$ determines a morphism
$$t_{W'}:F\rightarrow \mathcal{I}_{Z}\otimes\mathcal{L},$$
and conversely, we recover $W'$ from the data of $t_{W'}$ up to a scalar as the space of sections of
${\rm Ker}\,t_{W'}\subset F$.
We thus have an injection of the fiber ${\phi^0}^{-1}(Z)$ into
$\mathbf{P}({\rm Hom}\,(F,\mathcal{I}_{Z}\otimes\mathcal{L}))$.

In order to compute ${\rm Hom}\,(F,\mathcal{I}_{Z}\otimes\mathcal{L})$, we
tensor by $F^*$ the exact sequence (\ref{exact}).
We then get the long exact sequence:
\begin{eqnarray}\ldots\rightarrow {\rm Hom}\,(F,F)\rightarrow {\rm Hom}\,(F,\mathcal{I}_{Z}\otimes\mathcal{L})\rightarrow H^1(X,F^*\otimes W).
\end{eqnarray}
By the vanishing (\ref{van}) we conclude that
the map $${\rm Hom}\,(F,F)\rightarrow {\rm Hom}\,(F,\mathcal{I}_{Z}\otimes\mathcal{L})$$
is surjective. As $F$ is simple, the left hand side is generated by $Id_F$, so the
right hand side is generated by $t_W$. The fiber ${\phi^0}^{-1}(Z)$ thus consists of one point.

\end{proof}

\begin{proof}[Proof of Theorem \ref{ogrady}] Let $F$ be a simple nontrivial
globally generated vector bundle of rank $r$, with $h^1(F)=0$ and
with Mukai vector
$$v=v_F=(r,l,s)\in H^*(X,\mathbf{Z}).$$
This means that $r={\rm rank}\,F$, $l=c_1^{top}(F)\in H^2(X,\mathbf{Z})$ and
\begin{eqnarray}\label{derder}\chi(X,{\rm End}\,F)=<v,v^*>=2rs-l^2.
\end{eqnarray}
The Riemann-Roch formula applied to ${\rm End} \,F$
gives
\begin{eqnarray}\label{der}\chi(X,{\rm End}\,F)=2r^2+(r-1)l^2-2rc_2^{top}(F),
\end{eqnarray}
hence we get the formula (which can be derived as well from the definition
(\ref{vraimukai})) :
\begin{eqnarray}\label{for0}s=r+\frac{l^2}{2}-c_2^{top}(F).
\end{eqnarray}
We also have by definition of $d(v)$
 $$\chi(X,{\rm End}\,F)=2-2d(v)$$
  and thus  by (\ref{derder}),
\begin{eqnarray}\label{for1}d(v)=1-rs+\frac{l^2}{2}.
\end{eqnarray}

The Riemann-Roch formula applied to $F$ gives on the other hand:
\begin{eqnarray}\label{for2}\chi(X,F)=2r+\frac{l^2}{2}-c_2^{top}(F)
\end{eqnarray}
which by (\ref{for0}) gives
\begin{eqnarray}\label{for3}\chi(X,F)=r+s.
\end{eqnarray}
As we assume $h^1(F)=0$ and we have $h^2(F)=0$ since $F$ is
nontrivial, generated by sections and simple, we thus get
\begin{eqnarray}\label{for3bis}h^0(X,F)=r+s.
\end{eqnarray}
With the notations introduced above, we conclude that
$${\rm dim}\,G=(r-1)(s+1).$$
By Proposition \ref{propinj}, as all cycles parameterized by ${\rm Im}\,\phi$
 are rationally equivalent in $X$, the orbit under rational equivalence of
$c_2(F)$ in $X^{(k)}$, $k=c_2^{top}(F)$, has dimension greater than or equal to
$$ (r-1)(s+1)=rs-s+r-1.$$

But we have by (\ref{for0}) and (\ref{for1}):
$$k-d(v)=r-s+rs-1.$$

By Theorem \ref{autreforme},  we conclude that
$c_2(F)\in S^k_{d(v)}(X)$.

\end{proof}
\begin{remark}{\rm Instead of proving that the general $Z_W$ is reduced and applying
Theorem \ref{autreforme}, we could as well apply directly the variant \ref{variant}
to the family of subschemes $Z_W$.}
\end{remark}
For completeness, we  conclude this section with the proof of
Corollary \ref{corogrady}, although a large part of it  mimics the
proof of Proposition \ref{propogrady} in \cite{ogrady}.

We recall for convenience the statement:
\begin{corollary} \label{corointrobis}Let $v\in H^*(X,\mathbf{Z})$ be a Mukai vector. Assume there exists a simple vector bundle
$F$
on $X$ with Mukai vector $v$. Then
$$S_d^k(X)= \{c_2({G}), \,\,G\,\,{\rm a\,\, simple\,\, vector\,\, bundle\,\,on}\,\,X,\,\,v_G=v\},$$
where $d=d(v),\,k=c_2(v):=c_2^{top}(F),\,v_F=v$.

\end{corollary}

\begin{proof} The inclusion
\begin{eqnarray}\label{*}\{c_2({G}), \,\,G\,\,{\rm a\,\, simple\,\, vector\,\,
bundle\,\,on}\,\,X,\,\,v_G=v\}\subset S_d^k(X)
\end{eqnarray}
is the content of Theorem
\ref{ogrady}.

For the reverse inclusion, we first prove that there exists a Zariski open set
$U\subset X^{(d)}$ such that
\begin{eqnarray}\label{inclusionU}
cl(U)+(k-d(v))c_X\subset \{c_2({G}), \,\,G\,\,{\rm a\,\, simple}\\
\nonumber\, {\rm vector\,\, bundle\,\,on}\,\,X,\,\,v_G=v\}
\end{eqnarray}
where $cl: X^{(d)}\rightarrow CH_0(X)$ is the cycle map.

As $F$ is simple, the local deformations of $F$ are unobstructed.
 Hence there exist a smooth connected quasi-projective variety
$Y$, a locally free  sheaf $\mathcal{F}$ on $Y\times X$ and a point
$y_0\in Y$ such that $\mathcal{F}_{y_0}\cong F$ and the Kodaira-Spencer map
$$\rho: T_{Y,y_0}\rightarrow H^1(X,{\rm End}\,F)$$
is an isomorphism.

As $\mathcal{F}_{y_0}$ is simple, so is $\mathcal{F}_y$ for $y$ in a
dense  Zariski open set  of $Y$. Shrinking $Y$ if necessary,
$\mathcal{F}_y$ is simple for all $y\in Y$. By Theorem \ref{ogrady},
we have $c_2(\mathcal{F}_y)\in S_{d(v)}(X)$ for all $y\in Y$.

Let $\Gamma:=c_2(\mathcal{F})\in CH^2(Y\times X)$.
Consider the following set $R\subset Y\times X^{(d(v))}$
$$R=\{(y,z),\,\Gamma_*(y)=c_2(\mathcal{F}_y)=cl(z)+(k-d(v))c_X\,\,{\rm in}\,\,CH_0(X)\},$$
where $cl: X^{(d(v))}\rightarrow CH_0(X)$ is the cycle map and $k=c_2(v)$.

$R$ is a countable union of closed algebraic subsets of
$Y\times X^{(d)}$ and by the above inclusion (\ref{*}), the first projection
$$R\rightarrow Y$$
is surjective. By a Baire category argument, it follows that for
some component $R_0\subset R$, the first projection is dominant.

We claim that the second projection $R_0\rightarrow X^{(d(v))}$ is
also dominant. This follows from the fact that by Mumford's theorem,
the pull-back to $R_0$ of the holomorphic $2$-forms on $Y$ and
$X^{(d(v))}$ are equal. As the first projection is dominant and the
Mukai form on $Y$ has rank $2d(v)$, the same is true for its
pull-back to $R_0$ (or rather its smooth locus). Hence the pull-back
to $R_0$ of the symplectic form on $X^{(d(v))}$ by the second
projection also has rank $2d(v)$. This implies that the second
projection is dominant hence that its image contains a Zariski open
set. Thus (\ref{inclusionU}) is proved. The proof of Corollary
\ref{corointrobis} is then concluded with Lemma \ref{lemmafinal}
below.

\end{proof}
\begin{lemma}\label{lemmafinal} Let $X$ be a $K3$ surface and $d>0$ be an integer.
Then for any open set (in the analytic or Zariski topology) $U\subset
X^{(d)}$, we have
$$cl(U)=cl(X^{(d)})\subset CH_0(X).$$

\end{lemma}
\begin{proof} It clearly suffices to prove the result for $d=1$. It is proved in
\cite{maclean} that for any point $x\in X$, the set of points $y\in X$ rationally equivalent to $x$ in $X$ is dense in $X$ for the usual topology. This set thus
meets $U$, so that $x\in cl(U)$.

\end{proof}

 Centre de math\'{e}matiques Laurent Schwartz

91128 Palaiseau C\'{e}dex

 France

\smallskip
 voisin@math.polytechnique.fr
    \end{document}